\documentclass[10pt]{amsart}
\usepackage{graphicx}
\usepackage{txfonts, amsfonts}
\usepackage{eucal}
\usepackage{amssymb}
\usepackage{amsthm}
\usepackage{amscd}
\usepackage{multicol}
\usepackage{amsmath, amsthm, mathrsfs, hyperref}
\hypersetup{linkcolor=blue, citecolor=green, colorlinks=true,
pdfauthor={Matthew Chasse}} 

\theoremstyle{plain}
\newtheorem{thm}{Theorem}[section]

\newtheorem{cor}[thm]{Corollary}
\newtheorem{lem}[thm]{Lemma}

\theoremstyle{definition}
\newtheorem{defn}[thm]{Definition}
\newtheorem{conj}[thm]{Conjecture}
\newtheorem{exmp}[thm]{Example}

\theoremstyle{remark}
\newtheorem{rem}[thm]{Remark}

\def\lp{\mathcal{L}\text{-}\mathcal{P}}
\def\rl{\mathbb{R}}
\def\cn{\mathbb{C}}
\newcommand{\bea}{\begin{eqnarray*}} 
\newcommand{\eea}{\end{eqnarray*}}  
\newcommand{\bnum}{\begin{enumerate}}
\newcommand{\enum}{\end{enumerate}}
\newcommand{\bit}{\begin{itemize}}  
\newcommand{\eit}{\end{itemize}}
\newcommand{\beq}{\begin{equation}}  
\newcommand{\eeq}{\end{equation}}
\renewcommand{\exists}{\text{there exists }}
\renewcommand{\forall}{\text{for all }}

\begin{document} 
\title[A conjecture of I. Krasikov]{Discrete analogues of the Laguerre inequalities and a conjecture of I. Krasikov}

\author[G.~Csordas]{George CSORDAS}
\address{Department of Mathematics,University of Hawaii, Honolulu, HI 96822 }
\email{george@math.hawaii.edu}

\author[M.~Chasse]{Matthew CHASSE}
\address{ Department of Mathematics,University of Hawaii, Honolulu, HI 96822 }
\email{chasse@math.hawaii.edu}

\keywords{Laguerre inequality,
discrete polynomials, orthogonal polynomials, Laguerre inequalities}
\subjclass[2000]{ Primary 26D05; Secondary 30C10 }

\begin{abstract}
A conjecture of I. Krasikov is proved.  
Several discrete analogues of classical polynomial inequalities are derived, along with results which allow extensions to a class of transcendental entire functions in the Laguerre-P\'olya class.
\end{abstract}

\maketitle

\section{ Introduction}
The classical Laguerre inequality for polynomials states that a polynomial of degree $n$ with only real zeros, $p(x)\in\rl[x]$, satisfies
$(n-1) p'(x)^2 - n p''(x)p(x)\ge 0$
for all $x\in\rl$ (see \cite{CC,Lo}).  Thus, the classical Laguerre inequality is a necessary condition for a polynomial to have only real zeros.  Our investigation is inspired by an interesting paper of I. Krasikov \cite{K1}.  He proves several discrete polynomial inequalities, including useful versions of generalized Laguerre inequalities \cite{P2}, and shows how to apply them by obtaining bounds on the zeros of some Krawtchouk polynomials.  In \cite{K1}, I. Krasikov conjectures a new discrete Laguerre inequality for polynomials.  After establishing this conjecture, we generalize the inequality to transcendental entire functions (of order $\rho<2$, and minimal type of order $\rho=2$) in the Laguerre-P\'olya class (see Definition \ref{lp}). 

\begin{defn}
A real entire function $\varphi(x)=\sum_{k=0}^{\infty}\frac{\gamma_k}{k!}x^k$ is said to belong to the 
\emph{Laguerre-P\'olya class}, written $\varphi\in \lp$, if it can be expressed in the form
$$\varphi(x) = cx^me^{-ax^2 + bx}\prod_{k=1}^\omega\left(1+\frac{x}{x_k}\right)e^{\frac{-x}{x_k}} \text{\hspace{5 mm}} (0\le\omega\le\infty),$$
where $b,c,x_k \in \mathbb{R}$, $m$ is a non-negative integer, $a\ge 0$, $x_k\neq 0$, and 
$\sum_{k=1}^\omega\frac{1}{x_k^2}<\infty$.
\label{lp}
\end{defn}

The significance of the Laguerre-P\'olya class stems from the fact that functions in this class, \emph{and only these}, are uniform limits, on compact subsets of $\cn$, of polynomials with only real zeros \cite[Chapter VIII]{L}.

\begin{defn}
We denote by $\lp_n$ the set of polynomials of degree $n$ in the Laguerre-P\'olya class; that is, $\lp_n$ is the set of polynomials of degree $n$ having only real zeros.
\end{defn}

The minimal spacing between neighboring zeros of a polynomial in $\lp_n$ is a scale that provides a natural criterion for the validity of discrete polynomial inequalities.   
\begin{defn}
Suppose $p(x)\in\lp_n$ has zeros $\{\alpha_k\}_{k=1}^n$, repeated according to their multiplicities, and ordered such that $\alpha_k\le \alpha_{k+1}$, $1\le  k\le n-1$.  We define the \emph{mesh size}, associated with the zeros of $p$, by 
$$\mu(p) := \min_{1\le k \le n-1} |\alpha_{k+1} - \alpha_{k}|.$$   
\end{defn}

With the above definition of mesh size, we can now state a conjecture of I. Krasikov, which is proved in Section \ref{kcnjpf}. 
\begin{conj}(I. Krasikov \cite{K1})
\label{kcnj}
If $p(x)\in\lp_n$ and $\mu(p)\ge 1$, then 
\beq
(n-1)[p(x+1)-p(x-1)]^2 - 4np(x)[p(x+1)-2p(x)+p(x-1)] \ge 0
\label{kineq}
\eeq
holds for all $x\in\rl$.
\end{conj}

The classical Laguerre inequality is found readily by differentiating the logarithmic derivative of a polynomial $p(x)$ with only real zeros $\{\alpha_i\}_{i=1}^n$, to give
\beq
\frac{p''(x)p(x) - (p'(x))^2}{(p(x))^2} = \left(\frac{p'(x)}{p(x)}\right)' = \left(\sum_{k=1}^n \frac{1}{(x-\alpha_k)}\right)' = - \sum_{k=1}^n \frac{1}{(x-\alpha_k)^2}. 
\label{lineq}
\eeq
 Since the right-hand side is non-positive, 
$$ (p'(x))^2 - p''(x)p(x) \ge 0. $$
  This inequality is also valid for an arbitrary function in $\lp$ \cite{CC}.  A sharpened form of the Laguerre inequality for polynomials can be obtained with the Cauchy-Schwarz inequality,
\beq
\left(\sum_{k=1}^n \frac{1}{(x-\alpha_k)}\right)^2 \le n\sum_{k=1}^n \frac{1}{(x-\alpha_k)^2}.
\label{csineq}
\eeq
In terms of $p$, \eqref{csineq} becomes $\left(\frac{p'(x)}{p(x)}\right)^2 \le n\sum_{k=1}^n \frac{1}{(x-\alpha_k)^2}$, and
with \eqref{lineq} yields
the sharpened version of the Laguerre inequality for polynomials on which Conjecture \ref{kcnj} is based,
\beq
 (n-1)(p'(x))^2 - np''(x)p(x) \ge 0.
\eeq

The inequality \eqref{kineq} is a finite difference version of the classical Laguerre inequality for polynomials. Indeed, let us define
\beq
f_n(x,h,p):=(n-1)[p(x+h)-p(x-h)]^2 - 4np(x)[p(x+h)-2p(x)+p(x-h)].
\label{fndef}
\eeq
Then \eqref{kineq} can be written as $f_n(x,1,p)\ge0$ ($x\in\rl$), and we recover the classical Laguerre inequality for polynomials by taking the following limit:
\bea
\lim_{h\to0} \frac{f_n(x,h,p)}{4h^2} &=&   (n-1)\left(\lim_{h\to0} \frac{p(x+h)-p(x-h)}{2h}\right)^2\\
                                   & & \text{ \hspace{5mm} } - np(x)\left( \lim_{h\to0} \frac{p(x+h)-2p(x)+p(x-h)}{h^2}\right)\\
&=& (n-1) p'(x)^2 - n p''(x)p(x).
\eea

As I. Krasikov points out, the motivation for inequalities of type \eqref{kineq} is that classical discrete orthogonal polynomials $p_k(x)$ satisfy a three-term difference equation (see \cite[p. 27]{N}, \cite{K1})
$$p_k(x+1)=b_k(x)p_k(x) - c_k(x)p_k(x-1),$$
 where $b_k(x)$ and $c_k(x)$ are continuous over the interval of orthogonality.  Many of the classical discrete orthogonal polynomials satisfy the condition that $c_k(x)>0$ on the interval of orthogonality, and this implies that $\mu(p)\ge 1$ (see \cite{K4}).  Therefore, inequalities when $\mu(p)\ge 1$ are of interest and may help provide sharp bounds on the loci of zeros of discrete orthogonal polynomials \cite{K1,FK,FK2}.  Indeed, W. H. Foster, I. Krasikov, and A. Zarkh have found bounds on the extreme zeros of many orthogonal polynomials using discrete and continuous Laguerre and new Laguerre type inequalities which they discovered \cite{FK,FK2,FK3,K1,K2,K3,K4}.

In this paper, we prove I. Krasikov's conjecture (see Theorem \ref{mainthm}), extend it to a class of transcendental entire functions in the Laguerre-P\'olya class, and formulate several conjectures  (cf. Conjecture \ref{zspcconj}, Conjecture \ref{msrconj}, Conjecture \ref{msrconjcon}, and Conjecture \ref{lpcnj}).  In Section \ref{kcnjpf}, we establish several preliminary results about polynomials which satisfy a zero spacing requirement. 
In Section \ref{extlp}, we establish the existence of a polynomial sequence which satisfies a zero spacing requirement and converges uniformly on compact subsets of $\cn$ to the exponential function.  We use this result to extend a version of \eqref{kineq} to transcendental entire functions in the Laguerre-P\'olya class up to order $\rho=2$ and minimal type, and conjecture that it is true for all functions in $\lp$.

\section{Proof of I. Krasikov's Conjecture} 
\label{kcnjpf}

In this section we develop some discrete analogues of classical inequalities, form some intuition about the effect of imposing a minimal zero spacing requirement on a polynomial in $\lp$, and prove Conjecture \ref{kcnj}.  First, note that one can change the zero spacing requirement in Conjecture \ref{kcnj} by simply rescaling in $x$.  For example, the following conjecture is equivalent to Conjecture \ref{kcnj} of Krasikov.

\begin{conj}
Let $p(x)\in\lp_n$.
Suppose that $\mu(p)\ge h > 0$. Then for all $x\in\rl$,
\beq                                                 
f_n(x,h,p) = (n - 1)[p(x + h) - p(x - h)]^2 - 4np(x)[p(x+h) -2p(x) + p(x-h)] \ge 0.
\eeq
\end{conj} 
For the sake of clarity, we will work with \eqref{kineq} directly ($h=1$), and keep in mind that we can always make statements about polynomials with an arbitrary positive minimal zero spacing by rescaling $p(x)$ (in other words ``measuring $x$ in units of $h$'').

\begin{lem}
A local minimum of a polynomial, $p(x)\in\lp_n$, with only real simple zeros, is negative.  Likewise, a local maximum of $p(x)$ is positive.
\label{nnm}
\end{lem}

\begin{proof}
Because $p(x)$ is a polynomial on $\rl$ with simple zeros, at a local minimum ($x_{min}$, $p(x_{min})$), we have that $p'(x_{min})=0$ and $p''(x_{min})> 0$ (because $p''(x_{min})=0$ would imply that $p'$ has a multiple zero at $x_{min}$ which is not possible).  The classical Laguerre inequality asserts that if $p(x)\in\lp$, then for all $x\in\rl$, $(p'(x))^2 - p''(x)p(x) \ge 0$.  At a local minimum this expression becomes $- p''(x_{min})p(x_{min}) \ge 0$.  Therefore, at a local minimum we have $p(x_{min}) \le 0$. Since the zeros of $p$ are simple, $p(x_{min})\neq 0$. Thus $p(x_{min})<0$. The second statement of the lemma can be proved the same way, or by considering $-p$ and using the first statement.
\end{proof}

A statement similar to Lemma \ref{nnm} is proved by G. Csordas and A. Escassut \cite[Theorem 5.1]{C} for a class of functions whose zeros lie in a horizontal strip about the real axis.
\renewcommand\theenumi{\roman{enumi}}

\begin{lem}
Let $p(x)\in\lp_n$, $n\ge2$, $\mu(p)\ge 1$. 
\begin{enumerate}
\item If $p(x-1)>p(x)$ and $p(x+1)>p(x)$, then $p(x)<0$. \label{i1}
\item If $p(x-1)<p(x)$ and $p(x+1)<p(x)$, then $p(x)>0$. \label{i2}
\end{enumerate}

\label{sgnlm}

\end{lem}

\begin{proof}
\eqref{i1} Fix an $x_0\in\mathbb{R}$.  Let $p(x_0-1)>p(x_0)$, $p(x_0+1)>p(x_0)$, and assume for a contradiction that $p(x_0)\ge 0$.  There cannot be any zeros of $p(x)$ in the interval $[x_0-1,x_0]$, for if there were, $p(x_0)p(x_0-1)>0$ implies that the number of zeros in $(x_0-1,x_0)$ must be even, and this violates the zero spacing $\mu(p)\ge1$.  Similarly, there cannot be any zeros of $p(x)$ in $[x_0,x_0+1]$.  If $p(x_0)<p(x_0-1)$ and $p(x_0)<p(x_0+1)$ then there is a point in $(x_0-1,x_0+1)$ where $p'$ changes sign from negative to positive.  This implies $p$ achieves a non-negative local minimum on $[x_0-1,x_0+1]$ which contradicts Lemma \ref{nnm}.   

\eqref{i2} The second statement follows by replacing $p$ with $-p$ in (\ref{i1}).  
\end{proof}

Using Lemma \ref{sgnlm} we can verify that if $p(x)<\min\{p(x+1),p(x-1)\}$, then $p(x)<0$ and thus the function 
\begin{eqnarray}
f_n(x,1,p) &=& (n-1)[p(x+1)-p(x-1)]^2 - 4np(x)[p(x+1)-2p(x)+p(x-1)] \nonumber \\
           &=& (n-1)[p(x+1)-p(x-1)]^2 \nonumber \\
           & & \quad - 4np(x)[(p(x+1)-p(x))+(p(x-1)-p(x))] \label{sgnprf}
\end{eqnarray}
has a non-negative second term and \eqref{kineq} is satisfied.  Similarly, \eqref{kineq} is valid when $p(x)>\max\{p(x-1),p(x+1)\}$. 
The proof of Conjecture \ref{kcnj} is now reduced to the case where $\min\{p(x+1),p(x-1)\}\le p(x) \le \max\{p(x+1),p(x-1)\}$.   It is easy to show that if for some $p(x)\in\lp_n$, $f_n(x,1,p)\ge0$ for all $x\in\rl$, then  for all $m\ge n$, $f_m(x,1,p)\ge0$ for all $x\in\rl$.  If $\mu(p)\ge 1$, but $m<\deg(p)$, then for some $x_0\in\rl$, $f_m(x_0,1,p)$ may be negative.  Indeed, let $p(x)=x(x-1)(x-2)$, then $f_3(x,1,p)=72(x-1)^2$ and $f_2(x,1,p)= -12(x-3)(x-1)^2(x+1)$. In particular, $f_2(4,1,p)=-540$.  

We next obtain inequalities and relations that are analogous to those used in deriving the continuous version of the classical Laguerre inequality for polynomials.

\begin{defn}Let $p(x)\in\lp_n$ have only simple real zeros $\{\alpha_k\}_{k=1}^n$.
Define forward and reverse ``discrete logarithmic derivatives'' associated with $p(x)$ by
\begin{eqnarray}
F(x)&:=& \frac{p(x+1) - p(x)}{p(x)} =: \sum_{k=1}^n \frac{A_k}{(x-\alpha_k)} \label{F}\\
 \text{and \hspace{8mm} } R(x)&:=& \frac{p(x) - p(x-1)}{p(x)} =: \sum_{k=1}^n \frac{B_k}{(x-\alpha_k)}. \label{R}
\label{FReqs}
\end{eqnarray}
Note that $\deg(p(x+1) - p(x))  < \deg(p(x))$ and $\deg(p(x) - p(x-1))  < \deg(p(x))$ permits unique partial fraction expansions of the rational functions $F$ and $R$.  Define the sequences $\{A_k\}_{k=1}^n$ and $\{B_k\}_{k=1}^n$ associated with $p(x)$ by requiring that they satisfy the equation above. 
\label{FRdf}
\end{defn}

\begin{rem}
For an arbitrary finite difference, $h$, the scaled versions of the functions in Definition \ref{FRdf} are $F(x):= \frac{p(x+h) - p(x)}{hp(x)}$ and $R(x):= \frac{p(x) - p(x-h)}{hp(x)}$.
\end{rem}

\begin{lem}
For $p(x)\in\lp_n$, $n\ge2$, with $\mu(p)\ge 1$ and zeros $\{\alpha_k\}_{k=1}^n$, the associated sequences $\{A_k\}_{k=1}^n$ and $\{B_k\}_{k=1}^n$ satisfy $A_k\ge 0$ and $B_k\ge 0$, for all $ k$, $1\le k \le n$.
\label{sgnABlem}
\end{lem}

\begin{proof}
From Definition \ref{FRdf} we have
$$ p(x+1) - p(x) = \sum_{k=1}^n \frac{A_k}{(x-\alpha_k)} p(x)  = \sum_{k=1}^n \left[A_k \prod_{j\neq k}(x-\alpha_j)\right].$$
Evaluating this at a zero of $p$ yields
$p(\alpha_k + 1) = A_k\prod_{j\neq k}(\alpha_k-\alpha_j) =  A_kp'(\alpha_k)$.

Thus, 
$$ A_k = \frac{p(\alpha_k +1)}{p'(\alpha_k)} 
\text{\hspace{3mm} and similarly \hspace{3mm} }
 B_k = \frac{- p(\alpha_k -1)}{p'(\alpha_k)}.$$

Since the zeros of $p$ are simple, for some neighborhood of $\alpha_k$, $U(\alpha_k)$,
\bea
x\in U(\alpha_k)\text{, }x<\alpha_k &\text{ implies }& p(x)p'(x)<0 \\
\text{ and \hspace{8mm}}x\in U(\alpha_k)\text{, }x>\alpha_k &\text{ implies }& p(x)p'(x)>0. 
\eea
Since the zeros are spaced at least 1 unit apart, $p(\alpha_k+1)$ is either $0$ or has the same sign as $p(x)$ for $x>\alpha_k$ on $U(\alpha_k)$. So for all $\varepsilon >0$ sufficiently small,  
$p(\alpha_k +1)p'(\alpha_k  + \varepsilon )\ge 0$, and by continuity $p(\alpha_k +1)p'(\alpha_k)\ge 0$.  Thus $A_k = \frac{p(\alpha_k +1)}{p'(\alpha_k)} \ge 0 $. Note $p'(\alpha_k)\neq 0$ since $\alpha_k$ is simple. Likewise, $p(\alpha_k -1)$ is either $0$ or has the same sign as $p'(x)$ for $x<\alpha_k$ on $U(\alpha_k)$.  Hence for all $\varepsilon >0$ sufficiently small, 
$p(\alpha_k -1)p'(\alpha_k  - \varepsilon )\le 0$.  By continuity, $p(\alpha_k -1)p'(\alpha_k)\le 0$, whence $B_k\ge0$.

\end{proof}

\begin{exmp}
If the zero spacing requirement in Lemma \ref{sgnABlem} is violated then some $A_k$ or $B_k$ may be negative.  Indeed, consider $p(x)=x(x+1-\varepsilon )$. Then 
$\frac{p(x+1)-p(x)}{p(x)} = \frac{A_1}{x} + \frac{A_2}{x+1-\varepsilon },$
where 
$$ A_1 = \frac{2-\varepsilon }{1-\varepsilon }  \text{ \hspace{5mm} } A_2 = \frac{-\varepsilon }{1-\varepsilon }. $$
For any positive $\varepsilon <1$, $\mu(p)= 1-\varepsilon $, and $A_2$ is negative.
\end{exmp}

\begin{cor}
For $p(x)\in\lp_n$, $n\ge2$, with $\mu(p)\ge 1$, the associated functions $F(x)$ and $R(x)$ (see Definition \ref{FRdf}) satisfy $F'(x)<0$ and $R'(x)<0$ on their respective domains.
\label{FRdec}
\end{cor}

\begin{proof}
This corollary is a direct result of differentiating the partial fraction expressions for $F$ and $R$ and applying Lemma \ref{sgnABlem}.
\end{proof}

Note that the degree of the numerator of $F(x)$ is $n-1$.  If  $\mu(p)\ge 1$, then $F(x)$ has $n-1$ real zeros, because $F(x)$ is strictly decreasing between any two consecutive poles of $F(x)$.  This proves the following lemma. 

\begin{lem}{\rm (P\'olya and Szeg\"o \cite[vol. II, p. 39]{PSz2})}
For $p(x)\in\lp_n$, $n\ge2$, with $\mu(p)\ge 1$, $F(x)$ and $R(x)$ have only real simple zeros.
\label{FRrlz}
\end{lem}
  
In the sequel (see Lemma \ref{spclem}), we show that if $\mu(p(x))\ge 1$, then $\mu(p(x+1)-p(x))\ge 1$, and the zeros of $F(x)$ and $R(x)$ are spaced at least one unit apart.

\begin{lem}
If $p(x)\in\lp_n$, then the associated sequences $\{A_k\}_{k=1}^n$ and $\{B_k\}_{k=1}^n$ satisfy $\sum_{k=1}^n A_k = n$ and  $\sum_{k=1}^n B_k = n$. 
\label{smABlm}
\end{lem}

\begin{proof}
  Let $p(x) = a_nx^n +a_{n-1}x^{n-1} + \cdots + a_0 \in\lp_n$ and denote the zeros of $p(x)$ by $\{\alpha_k\}_{k=1}^n$. 
 Observe that 
\beq
\lim_{|z|\to\infty} zF(z)=\lim_{|z|\to\infty} z\left( \frac{p(z+1)-p(z)}{p(z)}\right )=\lim_{|z|\to\infty}z\sum_{k=1}^n \frac{A_k}{(z-\alpha_k)}=\sum_{k=1}^n A_k.
\label{lim1}
\eeq
Then \eqref{lim1} and
\bea
p(z+1)-p(z) &=& a_n(z+1)^n + a_{n-1}(z+1)^{n-1} + \ldots + a_0 -[a_nz^n + a_{n-1}z^{n-1} + \ldots + a_0]\\
            &=& na_{n}z^{n-1} + O(z^{n-2}), \text{ } |z|\to\infty,
\eea
 imply that
$$\sum_{k=1}^n A_k = \lim_{|z|\to\infty} zF(z)= \lim_{|z|\to\infty} z\left( \frac{p(z+1)-p(z)}{p(z)}\right )= \lim_{|z|\to\infty} z\left( \frac{na_nz^{n-1} + O(z^{n-2})}{ a_nz^n +a_{n-1}z^{n-1} + \cdots + a_0)}\right) = n. $$
A similar argument shows that $\sum_{k=1}^n B_k = n$.
\end{proof}

\begin{lem}
Given $p(x)\in\lp_n$, $n\ge2$, with $\mu(p)\ge 1$, the associated functions $F(x)$ and $R(x)$ satisfy $(F(x))^2\le-nF'(x)$ and $(R(x))^2\le-nR'(x)$, for all $x\in\rl$, where $p(x)\neq 0$.
\label{cslem}
\end{lem}

\begin{proof}
From  Definition \ref{FRdf}, $F(x)=\sum_{k=1}^n \frac{A_k}{x-\alpha_k}$ and therefore $F'(x)=\sum_{k=1}^n \frac{-A_k}{(x-\alpha_k)^2}$.
By Lemma \ref{sgnABlem}, $\mu(p)\ge 1$ implies the constants $A_k\ge 0$.
Using the the Cauchy-Schwarz inequality,
$$(F(x))^2 = \left(\sum_{k=1}^n \frac{A_k}{x-\alpha_k}\right)^2 \le \left(\sum_{k=1}^n A_k\right) \sum_{k=1}^n \frac{A_k}{(x-\alpha_k)^2} = -nF'(x),$$
where Lemma \ref{smABlm} has been used in the last equality.  An identical argument shows 
$(R(x))^2\le -nR'(x)$ for all $x\in\rl$. 
\end{proof}

\begin{rem}
Simple examples show that the inequalities in Lemma \ref{cslem} are sharp (consider $p(x)=x(x+1-\varepsilon )$). 
\end{rem}

\begin{lem}
Let $p(x)\in\lp_n$, $n\ge2$, with $\mu(p)\ge 1$, and let  $\{\beta_k\}_{k=1}^{n-1}$ be the zeros of $p(x+1)-p(x)$.  Let $y\in\rl$ be such that $\min\{p(y+1),p(y-1)\}<p(y)<\max\{p(y+1),p(y-1)\}$. Then if the interval $[y-1,y]$ does not contain any $\beta_k$ ,  
$$ \frac{1}{n}F(y)R(y) \le \frac{(p(y))^2 - p(y+1)p(y-1)}{(p(y))^2}.$$
\label{FRineq}
\end{lem}

\begin{proof}
If no  $\beta_k$ is in $[y-1,y]$, then $\frac{F'(x)}{(F(x))^2} = \frac{(p'(x+1)p(x)-p(x+1)p'(x))(p(x))^2}{(p(x+1)-p(x))^2(p(x))^2}$ can be extended to be continuous and bounded on $[y-1,y]$.  By Lemma \ref{cslem} $(F(x))^2\le-nF'(x)$.  Dividing both sides of this inequality by $n(F(x))^2$ and integrating from $y-1$ to $y$ we have
$$\frac{1}{n}\le\frac{1}{F(y)}-\frac{1}{F(y-1)}=\frac{p(y)}{p(y+1)-p(y)} - \frac{p(y-1)}{p(y)-p(y-1)}.$$
Using $\min\{p(y+1),p(y)\}<p(y)<\max\{p(y+1),p(y-1)\}$, we have that either $p(y-1)<p(y)<p(y+1)$ or $p(y+1)<p(y)<p(y-1)$.  In both cases, $(p(y+1)-p(y))(p(y)-p(y-1))>0$  and therefore 
\bea
\frac{1}{n}(p(y+1)-p(y))(p(y)-p(y-1))&\le& p(y)(p(y)-p(y-1))- p(y-1)(p(y+1)-p(y))\\
 &\le& (p(y))^2 -p(y+1)p(y-1).
\eea
Dividing both sides by $(p(y))^2$ gives the result.  
\end{proof}

\begin{lem}
For $p(x)\in\lp_n$, the associated functions $F(x)$ and $R(x)$ from Definition \ref{FRdf} satisfy
$$F(x)R(x) = (F(x)-R(x)) + \frac{(p(x))^2 - p(x+1)p(x-1)}{(p(x))^2}$$
for all $x\in\rl$, where $p(x)\neq 0$.
\label{mltid}
\end{lem}

\begin{proof}
This lemma is verified by direct calculation using the definitions of $F(x)$ and $R(x)$ in terms of $p(x)$.
\end{proof}

\begin{lem}
Let $p(x)\in\lp_n$, $n\ge2$, with $\mu(p)\ge 1$.  
\bnum
\item If $p(\beta)=p(\beta+1)>0$, then for all $x\in(\beta,\beta+1)$, $p(x)>p(\beta)$ and $p(x)>\max\{p(x+1),p(x-1)\}$.\label{i}  
\item If $p(\beta)=p(\beta+1)<0$, then for all $x\in(\beta,\beta+1)$,  $p(x)<p(\beta)$ and $p(x)<\min\{p(x+1),p(x-1)\}$.
\label{ii}
\item  If $p(\beta)=p(\beta+1)=0$, then for all $x\in(\beta,\beta+1)$, either $p(x)>\max\{p(x+1),p(x-1)\}$ or $p(x)<\min\{p(x+1),p(x-1)\}$.
\label{iii}
\enum
\label{dsctzr}
\end{lem}

\begin{proof}
Note that by Lemma \ref{FRrlz}, any $\beta$ which satisfies $p(\beta)=p(\beta+1)$ under the hypotheses stated in Lemma \ref{dsctzr} must be real and simple since $\beta$ is a zero of $F(x)$.

 For case \eqref{i}, assume for a contradiction that $\exists x_0\in(\beta,\beta+1)$ such that $p(x_0)\le p(\beta)$.
There can not be any zeros of $p$ on $(\beta,\beta+1)$, if there were, $p(\beta)p(\beta+1)>0$ implies that $p(x)$ must have at least two zeros on $(\beta,\beta+1)$, which contradicts $\mu(p)\ge1$.  Thus, for all $ x\in(\beta,\beta+1)$, $p(x)>0$.  Specifically $p(x_0)>0$.

Since $p(x)$ does not change sign on $(\beta,\beta+1)$, the interval $(\beta,\beta+1)$ must lie between two neighboring zeros of $p(x)$, call them $\alpha_1$ and $\alpha_2$, such that $(\beta,\beta+1)\subset(\alpha_1,\alpha_2)$.  By the mean value theorem $\exists a\in(\beta,\beta+1)$ with $p'(a)=0$.  The zeros of $p(x)$ and $p'(x)$ interlace, and in order to preserve the interlacing $a$ must be the only zero of $p'(x)$ in $(\alpha_1,\alpha_2)$, hence $p'(\beta),p'(\beta + 1)\neq 0$.  Because the zeros are simple, for some $\varepsilon >0$, for all $ x\in(\alpha_1,\alpha_1+\varepsilon )$, $p'(x)p(x)>0$, and for all $ x\in(\alpha_2-\varepsilon ,\alpha_2)$, $p'(x)p(x)<0$.  Since $p'$ and $p$ do not change sign on $(\alpha_1,\beta)$ or $(\beta +1, \alpha_2)$, this gives us that $p'(\beta)>0$ and $p'(\beta+1)<0$. Then if $p(x_0)\le p(\beta)$, $p'$ must change signs at least twice  on $(\alpha_1,\alpha_2)$ (actually three times), at least once on $(\beta,x_0)$ and at least once on $(x_0,\beta+1)$, and this contradicts the uniqueness of $a$.  Thus for all $ x\in(\beta,\beta+1)$ we have $p(x)>p(\beta)$.

To show $p(x)>p(\beta)$ implies $p(x)>\max\{p(x+1),p(x-1)\}$ for all $ x\in(\beta,\beta+1)$, notice that since $p'(y)<0$ for all $y\in(\beta+1,\alpha_2)$, $p(\beta+1)>p(y)$ for all $y\in(\beta+1,\alpha_2)$, and due to the zero spacing $p\le 0$ on $(\alpha_2, \alpha_2+1)$, hence $p(\beta+1)>p(x+1)$ for all $ x\in(\beta,\alpha_2)$.  Thus, for all $x\in(\beta,\beta+1)$, $p(x)>p(\beta +1)>p(x+1)$.  In the same way, $p'(y)>0$ for $y\in(\alpha_1, \beta)$ and $p\le0$ on $(\alpha_1-1,\beta)$ imply that $p(\beta)>p(x)$ for all $ x\in(\alpha_1-1,\beta)$ and therefore $p(x)>p(x-1)$ for all $x\in(\beta, \beta+1)$.  Hence, for all $ x\in(\beta,\beta+1)$, $p(x)>p(x-1)$ and $p(x)>p(x+1)$, therefore $p(x)>\max\{p(x+1),p(x-1)\}$. 

Consider case \eqref{iii}.  If $p(\beta) = p(\beta+1) = 0$, then $p$ does not change sign on $(\beta,\beta+1)$ since $\mu(p)\ge1$.  It suffices to consider the case when $p$ is positive on $(\beta,\beta+1)$.  Then for all $x\in (\beta,\beta+1)$, $p(x)>0=p(\beta)$.  The conclusion  $p(x)>\max\{p(x+1),p(x-1)\}$ ($p(x)<\min\{p(x+1),p(x-1)\}$) is a consequence of $p(x)>p(\beta)$ ($p(x)<p(\beta)$) by the same argument given in the proof of case (\ref{i}).

To prove \eqref{ii}, let $g(x)=-p(x)$ and apply \eqref{i}.
 
\end{proof}

\begin{lem}
If $p(x)\in\lp_n$, $n\ge2$, $\mu(p)\ge 1$, and $g(x)=p(x+1)-p(x)$, then $\mu(g)\ge 1$.
\label{spclem}
\end{lem}

\begin{proof}({\it Reductio ad Absurdum})
If $\mu(g)<1$, then there exist $\beta_1,\beta_2\in\rl$ such that $0<\beta_2-\beta_1<1$ and $g(\beta_1)=g(\beta_2)=0$.  In the proof of Lemma \ref{dsctzr} we have shown that $p(x)$ does not change sign on $(\beta_1, \beta_1 +1)$.  Without loss of generality assume that $p$ is positive on $(\beta_1,\beta_1 + 1)$. Observe that $\beta_2\in(\beta_1,\beta_1+1)$, and thus by Lemma \ref{dsctzr}, $p(\beta_2)>\max\{p(\beta_2+1),p(\beta_2-1)\}\ge p(\beta_2+1)$.  But this yields $p(\beta_2+1)-p(\beta_2)<0$, and therefore $g(\beta_2)<0$ contradicting $g(\beta_2)=0$.
\end{proof}

Note that Lemma \ref{spclem} is equivalent to the statement that if $p(x)\in\lp_n$ with $\mu(p)\ge 1$, then the associated functions $F(x)$ and $R(x)$ also have zeros spaced at least $1$ unit apart.  Preliminaries aside, we prove Conjecture \ref{kcnj} of I. Krasikov.

\begin{thm}
If $p(x)\in\lp_n$ and $\mu(p)\ge 1$, then 
\beq
f_n(x,1,p) = (n-1)[p(x+1)-p(x-1)]^2 - 4np(x)[p(x+1)-2p(x)+p(x-1)] \ge 0
\label{k2}
\eeq
holds for all $x\in\rl$.
\label{mainthm}
\end{thm}

\begin{proof}
Since \eqref{k2} is true when ${\rm deg}(p(x))$ is $1$ or $2$, we assume $n\ge 2$.
Fix $x=x_0\in\rl$.  If $p(x_0-1)=p(x_0)=p(x_0+1)$, or if $p(x_0)=0$, then $f_n(x,1,p)\ge 0$.  Thus, we may assume $p(x_0)\neq 0$. If $p(x_0)<\min\{p(x_0+1),p(x_0-1)\}$, or if $p(x_0)>\max\{p(x_0+1),p(x_0-1)\}$, then $f_n(x_0,1,p)\ge 0$ (use \eqref{sgnprf} and Lemma \ref{sgnlm}).    

We next consider the case when
 \beq
\min\{p(x_0-1),p(x_0+1)\}< p(x_0) <\max\{p(x_0-1),p(x_0+1)\} \label{pnm}
\eeq 
(thus $x_0\neq\beta \text{ or } \beta + 1$, where $p(\beta+1)=p(\beta)$), and show
$$\frac{f_n(x_0,1,p)}{(p(x_0))^2} = (n-1)(F(x_0)+R(x_0))^2 - 4n(F(x_0)-R(x_0)) \ge 0,$$
where $F(x)$ and $R(x)$ are defined by \eqref{F} and \eqref{R} respectively.
By Lemma \ref{mltid},
\begin{eqnarray}
\frac{f_n(x_0,1,p)}{(p(x_0))^2} &=& (n-1)(F(x_0)-R(x_0))^2\nonumber\\
                                & & \quad- 4n\left(\frac{1}{n}F(x_0)R(x_0)-\frac{(p(x_0))^2 - p(x_0+1)p(x_0-1)}{(p(x_0))^2}\right).
\label{factored}
\end{eqnarray}

By Lemma \ref{spclem}, $\mu(p(x+1)-p(x))\ge1$, and thus the zeros $\{\beta_k\}_{k=1}^{n-1}$ of $F(x)$ ($p(\beta_k+1)=p(\beta_k)$) are spaced at least one unit apart.  If $[x_0-1,x_0]$ does not contain any $\beta_k$, $\frac{f_n(x_0,1,p)}{(p(x_0))^2}\ge 0$ holds by Lemma \ref{FRineq} (see \eqref{factored}) .
If, on the other hand, $\beta_j\in(x_0-1,x_0)$ (recall $\beta_j\neq x_0, x_0-1$), then $x_0\in (\beta_j, \beta_j+1)$ and by Lemma \ref{dsctzr} either $p(x_0)>\max\{p(x_0-1),p(x_0+1)\}$ or $p(x_0)<\min\{p(x_0-1),p(x_0+1)\}$, and both of these cases contradict our assumption (see \eqref{pnm}).  We have now shown $f_n(x_0,1,p))\ge0$ for all $x_0\in\rl$, except for the isolated points where $x_0=\beta_j$ or $x_0=\beta_j+1$ for some $j$, but by continuity of $f_n(x,1,p)$, \eqref{k2} will hold.  

\end{proof}

The converse of Theorem \ref{mainthm} is false in general.  Indeed, the following example shows that there are polynomials with arbitrary minimal zero spacing that still satisfy $f_n(x,1,p)\ge 0$ for all $x\in\rl$.

\begin{exmp} Let $p(x)= (x+n+a)\prod_{k = 1}^{n-1} (x+k)$  with $n\ge 2$, $a\in\rl$. 
Using a symbolic manipulator (we used Maple)
$$f_n(x,1,p) = C(x,n,a)\prod_{k=2}^{n-2} (x+k)^2$$
where
\begin{multline} 
C(x,n,a) := (n-1)(-2n^3-4na+4a^2+n^2+n^4)x^2 \\
                +(n-1)(6n^2a+4n^4-8n^3a+8a^2-12na+4na^2-8n^3+2n^4a+4n^2)x\\
                +(n-1)(-8na-4na^2+4a^2+4n^4a-8n^3+4n^4+4n^2+12n^2a\\
+n^4a^2+13n^2a^2-16n^3a-6n^3a^2).
\end{multline}
$C(x,n,a)$ is quadratic in $x$ and its discriminant is 
$D=-16na^2(n-1)^2(n-2)^3(a-n)^2\le0$.
Therefore $C(x,n,a)$ does not change sign and is always positive (this is verified by showing that the coefficient of $x^2$ is positive when considered as a quadratic in $a$),
whence $f_n(x,1,p) \ge 0$ for all $x\in\rl$. 
\end{exmp}

In general, a polynomial $p$ may satisfy $f_n(p,1,x)\ge0$ for all $x\in\rl$, even if $p$ has multiple zeros.
If $p(x)=x^2(x+1)$, which has $\mu(p)=0$, then $f_3(x,1,p)=56x^2+32x+8$ is non-negative for all $x\in\rl$. 
A polynomial $p$ with non-real zeros may also satisfy $f_n(p,1,x)\ge0$ for all $x\in\rl$.
For example, let $p(x)=(x^2+1)(x+1)$, then $f_3(x,1,p)=32x^2-32x+8\ge 0$ for all $x\in\rl$.   

It is known that a polynomial $p(x)\in\lp_n$ with only real zeros satisfies $\mu(p)\le\mu(p')$; that is, $p'(x)$ will have a minimal zero spacing which is larger than that of $p(x)$ (N. Obreschkoff \cite[p. 13, Satz 5.3]{O}, P. Walker \cite{W}).  In light of Lemma \ref{spclem}, the aforementioned result suggests the following conjecture.
\begin{conj}
If $p(x)\in\lp_n$, $n\ge2$, $\mu(p)\ge d \ge 1$, and $g(x)=p(x+1)-p(x)$, then $\mu(g)\ge d$. 
\label{zspcconj}
\end{conj}
The derivation of the classical Laguerre inequality relies on properties of the logarithmic derivative of a polynomial. In the same way, Conjecture \ref{kcnj} was proved using a discrete version of the logarithmic derivative.  The analogy between the discrete and continuous logarithmic derivatives motivates the following conjectures, based on Theorem \ref{ctmsr} and its converse (B. Muranaka \cite{M}). 

\begin{thm}{\rm(P. B. Borwein and T. Erd\'{e}lyi \cite[p. 345]{BE})}
If $p\in\lp_n$, then
$$ m\left( \left\{x\in\rl : \frac{p'(x)}{p(x)}\ge \lambda\right\} \right) 
= \frac{n}{\lambda} \text{ \hspace{5mm} } \forall\lambda>0,$$
where $m$ denotes Lebesgue measure.
\label{ctmsr}
\end{thm}

\begin{conj}
If $p\in\lp_n$, $n\ge2$, $\mu(p)\ge 1$, then
$$ m\left( \left\{x\in\rl : \frac{p(x+1)-p(x)}{p(x)}\ge \lambda\right\} \right) 
= \frac{n}{\lambda} \text{ \hspace{5mm} } \forall\lambda>0,$$
where $m$ denotes Lebesgue measure.
\label{msrconj}
\end{conj}

\begin{conj}
If $p(x)$ is a real polynomial of degree $n\ge 2$, and if
 $$ m\left( \left\{x\in\rl : \frac{p(x+1)-p(x)}{p(x)}\ge \lambda\right\} \right) 
= \frac{n}{\lambda} \text{ \hspace{5mm} } \forall\lambda>0,$$
where $m$ denotes Lebesgue measure, then $p\in\lp_n$ with $\mu(p)\ge 1$.
\label{msrconjcon}
\end{conj}

\section{Extension to a Class of Transcendental Entire Functions}
\label{extlp}

In analogy with \eqref{fndef} we define, for a real entire function $\varphi$,
\beq
f_{\infty}(x,h,\varphi):=[\varphi(x+h)-\varphi(x-h)]^2 - 4\varphi(x)[\varphi(x+h)-2\varphi(x)+\varphi(x-h)].
\label{trdli}
\eeq
For $\varphi\in\lp$, with zeros $\{\alpha_i\}_{i=1}^{\omega}$, $\omega\le\infty$, we introduce the mesh size 
\beq
\mu_\infty(\varphi) := \inf_{i\ne j} |\alpha_i - \alpha_j|.
\eeq
We remark that if $\psi\notin\lp$, then $\psi$ need not satisfy $f_{\infty}(x,h,\psi)\ge 0$ for all $x\in\rl$.
A calculation shows that if $\psi(x)=e^{x^2}$, then $f_{\infty}(0,1,\psi) = -8(e-1) <0$.
When $\varphi\in\lp_n$, $f_{\infty}(x,h,\varphi)\ge 0$ for all $x\in\rl$ by Theorem \ref{mainthm}.  In order to extend Theorem \ref{mainthm} to transcendental entire functions,
we require the following preparatory result to ensure that the approximating polynomials we use will satisfy a zero spacing condition.

\begin{lem}
For any $a\in\rl$, $n\in\mathbb{N}$, $n\ge 2$,
 $$\lim_{n\to\infty} \sum_{k=1}^{n^n} \frac{1}{n\ln(n)(k+n) + a} = 1.$$
\label{sumlem}
\end{lem}

\begin{proof}Fix $a\in\rl$.
Since the terms $\frac{1}{n\ln(n)(k+n) + a}$ are decreasing with 
$k$ for $n$ sufficiently large, we obtain

$$\int_1^{n^n+1} \frac{1}{n\ln(n)(k+n) + a} dk \le \sum_{k=1}^{n^n} \frac{1}{n\ln(n)(k+n) + a} \le \int_0^{n^n} \frac{1}{n\ln(n)(k+n) + a} dk,$$
\noindent
for $n$ sufficiently large, by considering the approximating Riemann sums for the integrals.  Thus

\beq
\frac{1}{n\ln(n)}\ln\left(\frac{n^n+1+\frac{a}{n\ln(n)}}{n+1+\frac{a}{n\ln(n)}}\right) \le \sum_{k=1}^{n^n} \frac{1}{n\ln(n)(k+n) + a} \le \frac{1}{n\ln(n)}\ln\left(\frac{n^n+\frac{a}{n\ln(n)}}{n+\frac{a}{n\ln(n)}}\right).\label{starry}
\eeq
As $n\to\infty$, both the left and right sides of \eqref{starry} approach $1$, and whence the sum in the middle approaches $1$.
\end{proof}

\begin{lem}
The set of polynomials $\left\{q_n(x) = \prod_{k=1}^{n^n}\left( 1 + \frac{x}{n\ln(n)(k+n)}\right) \text{:} n\in\mathbb{N}\text{, }n\ge 2\right\}$, forms a normal family on $\mathbb{C}$.  There is a subsequence of $\{q_n(x)\}_{n=2}^\infty$ which converges uniformly on compact subsets of $\cn$ to $e^x$.   
\label{simpPolys}
\end{lem}

\begin{proof}
Let $K\subset\cn$ be any compact set and let $R = \sup_{z\in K}|z|$.  Recall the inequality 
$$\frac{1}{2}|z| \le |\ln(1+z)| \le \frac{3}{2}|z| \text{ \hspace{5mm} for } |z | < \frac{1}{2}$$ \cite[p. 165]{CJ}.  
Then for $n > 2R$, $\left|\frac{z}{n\ln(n)(k+n)}\right|<\frac{1}{2}$,  hence, for $k\ge1$ and $z\in K$
$$\frac{1}{2}\frac{|z|}{n\ln(n)(k+n)} \le \left|\ln\left(1 + \frac{z}{n\ln(n)(k+n)} \right)\right| \le \frac{3}{2}\frac{|z|}{n\ln(n)(k+n)},$$
and therefore
$$\frac{1}{2}\sum_{k=1}^{n^n}\frac{|z|}{n\ln(n)(k+n)} \le \sum_{k=1}^{n^n}\left|\ln\left(1 + \frac{z}{n\ln(n)(k+n)} \right)\right| \le \frac{3}{2}\sum_{k=1}^{n^n}\frac{|z|}{n\ln(n)(k+n)}.$$
As $n\to\infty$ the sums on the left and right sides of the inequality converge by Lemma \ref{sumlem} to $\frac{1}{2}|z|$ and  $\frac{3}{2}|z|$  respectively.  In particular, for some $\varepsilon >0$  and $N>2R$ sufficiently large, for all $n\ge N$ and for all $z\in K$,
$$ \sum_{k=1}^{n^n}\left|\ln\left(1 + \frac{z}{n\ln(n)(k+n)} \right)\right| \le \frac{3}{2}R + \varepsilon .$$
Then for all $n\ge N$, for all $z\in K$,
$$\left| q_n(z) \right| 
\le  e^{\sum_{k=1}^{n^n}\left|\ln\left(1 + \frac{z}{n\ln(n)(k+n)} \right)\right|} 
\le e^{\frac{3}{2}R + \varepsilon }.$$
So for $n>N$ sufficiently large, the sequence $\{q_n(z)\}_{n=2}^\infty$ is uniformly bounded on compact subsets $K\subset\cn$ and thus form a normal family by Montel's theorem \cite[p. 153]{CJ}.  Thus, there is a subsequence of $\{q_n(z)\}_{n=2}^\infty$ which converges uniformly on compact subsets of $\cn$ to a function $f$, and therefore satisfies 
\beq
\frac{f'(x)}{f(x)} = \lim_{n\to\infty} \frac{q'_n(x)}{q_n(x)} 
= \lim_{n\to\infty} \sum_{k=1}^{n^n} \frac{1}{n\ln(n)(k+n) + x} = 1, \label{lgdr}
\eeq
for a fixed $x\in\rl$, where the last equality is by Lemma \ref{sumlem}.    
Equation \eqref{lgdr} and $f(0) = 1$, imply $f(x) = e^x$ on $\rl$, and thus $f$ is the exponential function.
\end{proof}

\begin{lem}
If $\varphi(x)=p(x)e^{b x}$, $b\in\rl$, $p\in\lp_n$, $n\ge2$, and $\mu(p)\ge1$, then $f_{\infty}(x,1,\varphi)\ge0$ for all $x\in\rl$.
\label{pexp}
\end{lem}

\begin{proof}
By Lemma \ref{simpPolys}, there is a subsequence of 
$\left\{q_j(x) = \prod_{k=1}^{j^j}\left(1 + \frac{x}{j\ln(j)(k+j)}\right)\right\}_{j=2}^\infty$,
call it $\{q_{j_m}(x)\}_{m=1}^\infty$, 
such that $q_{j_m}(x)\to e^{x}$ uniformly on compact subsets of $\cn$, as $m\to\infty$.
Let $\{\alpha_k\}_{k=1}^n$ be the zeros of $p(x)$, and 
$\displaystyle R=\max\limits_{1\le k\le n} |\alpha_k|.$
The zero of least magnitude of $q_{j_m}(bx)$, $z_{j_m}$, satisfies $|z_{j_m}|=\frac{j_m\ln(j_m)(1+j_m)}{b}$, $b\neq 0$.  
Both $\mu(q_{j_m}(bx))\to\infty$ as $m\to\infty$ 
and $|z_{j_m}|\to\infty$ as $m\to\infty$.  Thus, there is an $M$ such that for all $m>M$,  $|z_{j_m}| > R + 1$, and the sequence of polynomials $h_m(x) = p(x)q_{j_{M+m}}(bx)$, $m\ge 1$, is in $\lp_\ell$ for some $\ell$, and satisfies $\mu(h_m)\ge 1$.  By Theorem \ref{mainthm}, $f_\infty(x,1,h_m)\ge 0$ for all $x\in\rl$, for all $m$. Since $h_m\to p(x)e^{bx}$ by construction, $\lim_{m\to\infty} f_\infty(x,1,h_m)= f_\infty(x,1,p(x)e^{bx}) \ge 0$.  
\end{proof}

\begin{thm}
If $\varphi\in\lp$ has order $\rho<2$, or if $\varphi$ is of minimal type of order $\rho=2$,  and $\mu_{\infty}(\varphi)\ge 1$, then
$f_{\infty}(x,1,\varphi)\ge 0$ for all $x\in\rl$.
\label{ccte}
\end{thm}

\begin{proof}
By the Hadamard factorization theorem, $\varphi$ has the representation
$$\varphi(x) = cx^me^{bx}\prod_{k=1}^\omega \left(1+\frac{x}{a_k}\right)e^{-\frac{x}
{a_k}}  \text{ \hspace{5mm} }(\omega\le\infty),$$
where $a_k,b,c\in\rl$, $m$ is a non-negative integer, $a_k\neq 0$, and $\sum_{k=1}^\omega \frac{1}{a_k^2}<\infty$.  Let 
$$g_n(x) = cx^me^{bx}\prod_{k=1}^n \left(1+\frac{x}{a_k}\right)e^{-\frac{x}{a_k}}.$$
Then, $g_n(x) = ce^{bx-\sum_{k=1}^n\frac{x}{a_k}}x^m\prod_{k=1}^n \left(1+\frac{x}{a_k}\right)$ has the form $p(x)e^{\gamma x}$, $\gamma\in\rl$, $p\in\lp_n$, and thus by Lemma \ref{pexp}, $f_\infty(x,1,g_n)\ge 0$ for all $x\in\rl$, and for all $n$.  Since we also have $g_n\to \varphi$ by construction,   
$\lim_{n\to\infty}f_\infty(x,1,g_n) = f_\infty(x,1,\varphi)\ge 0$ for all $x\in\rl$. 
\end{proof}
In light of Theorem \ref{ccte}, we make the following conjecture.  
\begin{conj}
If $\varphi\in\lp$ and $\mu_{\infty}(\varphi) \ge 1$ then
$f_{\infty}(x,1,\varphi)\ge 0$ for all $x\in\rl$.
\label{lpcnj}
\end{conj}

\renewcommand\refname{References} 

%

\end{document}